\pdfoutput=1
\documentclass[11pt]{article}
\usepackage{amsmath,amsthm}
\usepackage{amssymb,latexsym}
\usepackage{graphicx}
\usepackage{mathpazo}
\usepackage{fullpage}
\usepackage{hyperref}
\usepackage{amsfonts}
\usepackage{enumerate}
\usepackage{ifpdf}
\usepackage{hyperref}
\newtheorem{theorem}{Theorem}[section]

\newtheorem{claim}[theorem]{Claim}
\newtheorem{lemma}[theorem]{Lemma}

\def\R{\mathbb{R}}

\def\mod{\rm{mod}}

\newcommand{\eps}{\varepsilon}
\renewcommand{\epsilon}{\varepsilon}

\begin{document}

\title{\bf Entropy-based Bounds on Dimension Reduction in~$L_1$}

\author{
 Oded Regev \footnote{Blavatnik School of Computer Science, Tel Aviv University, and
CNRS, D{\'e}partement d'Informatique, {\'E}cole normale sup{\'e}rieure, Paris.
Supported by the Israel Science Foundation and by a European Research Council (ERC) Starting Grant.
   }}
%   {\tt odedr@cs.tau.ac.il}. } }

%\date{}

\maketitle

\begin{abstract}
We show that for every large enough integer $N$, there exists an $N$-point subset of $L_1$ such that for every $D>1$,
embedding it into $\ell_1^d$ with distortion $D$ requires dimension $d$ at least $N^{\Omega(1/D^2)}$,
and that for every $\eps>0$ and large enough integer $N$, there exists
an $N$-point subset of $L_1$ such that embedding it into $\ell_1^d$ with distortion $1+\eps$ requires
dimension $d$ at least $N^{1-O(1/\log(1/\eps))}$.
These results were previously proven by Brinkman and Charikar [JACM, 2005]
and by Andoni, Charikar, Neiman, and Nguyen [FOCS 2011].
We provide an alternative and arguably more intuitive proof based on an entropy argument.
\end{abstract}

\section{Introduction}

We prove the following theorem.

\begin{theorem}\label{thm:mainthm}
For every large enough integer $N$, there exists an $N$-point subset of $L_1$ such that for every $D>1$,
embedding it into $\ell_1^d$ with distortion $D$ requires dimension $d$ at least $N^{\Omega(1/D^2)}$.
Moreover, for every $\eps>0$ and large enough integer $N$, there exists
an $N$-point subset of $L_1$ such that embedding it into $\ell_1^d$ with distortion $1+\eps$ requires
dimension $d$ at least $N^{1-O(1/\log(1/\eps))}$.
\end{theorem}

Both parts of Theorem~\ref{thm:mainthm} were previously known.
The first part (embedding with large distortion)
was first shown by Brinkman and Charikar~\cite{BrinkmanC05},
and later with a simpler proof by Lee and Naor~\cite{LeeNaor04}.
The second part (embedding with low distortion) was recently shown by
Andoni, Charikar, Neiman, and Nguyen~\cite{AndoniCNN10}.
Our proof is based on an entropy argument, and is arguably more intuitive.

The set of points we use is identical to the one used by Andoni et al.~\cite{AndoniCNN10}. For completeness,
we briefly describe it here (see also Figure~\ref{fig:graph} for an illustration).
For integers $k \ge 2$, $n \ge 1$, we define the so-called ``recursive cycle" graph $G_{k,n}$, and
associate with each vertex a label in $\{0,1\}^{k^n}$. The set of all labels will be our point set
$P_{k,n}$ in $\ell_1$. First, for $k \ge 2$, let $G_{k,1}$ be the cycle of length $2k$, with two
distinguished antipodal vertices (i.e., of distance $k$), call them ``left" and ``right".
For $0\le i \le k$, the $i$th vertex on the top path from the left to the right vertex is labeled
with the vector $(0,\ldots,0,1,\ldots,1)$ with $k-i$ zeros and $i$ ones, and the $i$th vertex on the
bottom path is associated with the vector $(1,\ldots,1,0,\ldots,0)$ with $i$ ones and $k-i$ zeros. Notice that the $\ell_1$
distance between the labels of any two adjacent vertices is $1$, whereas that between the labels of any two antipodal vertices is $k$.

For $n \ge 2$, define $G_{k,n}$ as the graph obtained from $G_{k,n-1}$ by replacing each edge with a copy of $G_{k,1}$
and identifying the distinguished vertices with the original endpoints of the edge.
%(We remark that the graphs $G_{2,n}$ are known as the diamond graphs.)
The number of vertices in $G_{k,n}$ is easily seen to be
$$N_{k,n}:=\frac{(2k-2)(2k)^{n}+2k}{2k-1} \le (2k)^n.$$
For the labels, we first take the labels in $G_{k,n-1}$ and duplicate each coordinate $k$ times. This defines the labels
for those vertices coming from $G_{k,n-1}$. For the newly added vertices on each cycle that replaced an edge
of $G_{k,n-1}$, we replace the $k$ coordinates on which the two distinguished nodes of that cycle differ with
the same labeling of $G_{k,1}$ described earlier. Notice the following two properties: the $\ell_1$ distance between
the labels of any two adjacent vertices is $1$, and for $1\le \ell\le n$, the distance
between any two antipodal vertices in level $\ell$ is $k^{n-\ell+1}$.
We remark that these two properties are also satisfied by the shortest path metric on $G_{k,n}$, but since that metric is not in $\ell_1$, it is not good enough for the purpose
of proving dimension reduction in $\ell_1$.

Finally, we label the edges of $G_{k,1}$ by elements of $[2k]$ starting from the left vertex and going along the cycle,
and extend this to a labeling of $G_{k,n}$ by elements of $[2k]^n$ in a recursive way, with
the coordinates labeling the location of the edge from the top layer to the bottom layer (see Figure~\ref{fig:graph}).

%It is known that $G_{k,n}$ embeds in $L_1$ with constant distortion~\cite{GuptaNRS04}.

\begin{figure}
\begin{centering}
\ifpdf
\includegraphics[width=0.8\textwidth]{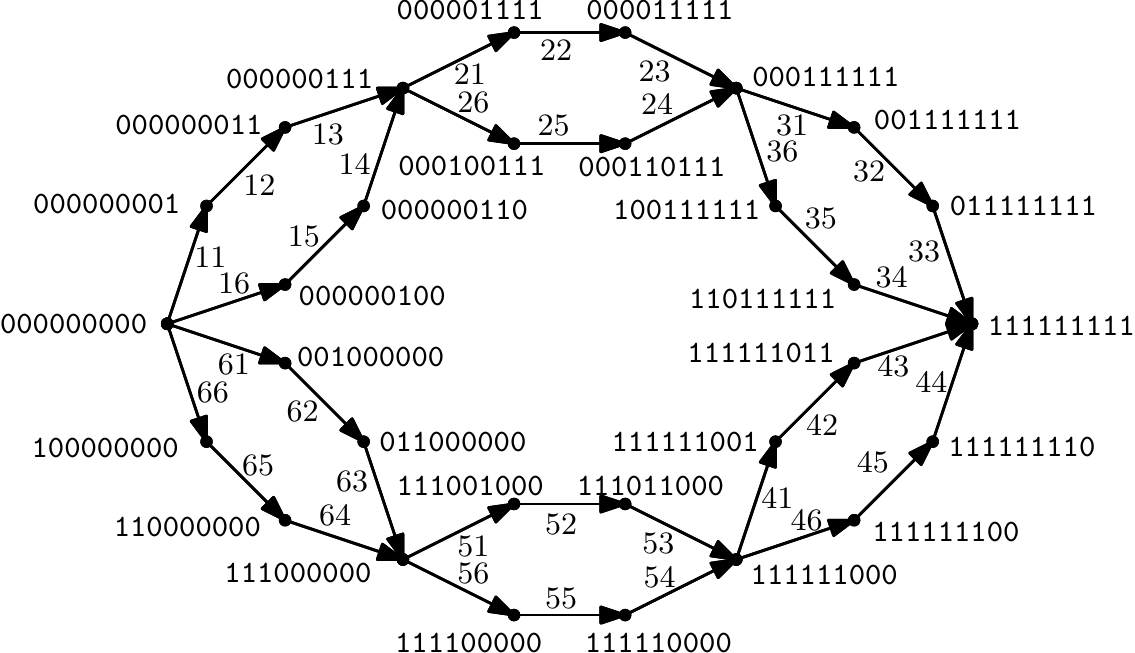}
\fi
\caption{$G_{3,2}$ with our labeling and orientation of the edges and the labels on vertices in $\{0,1\}^9$.}
\label{fig:graph}
\end{centering}
\end{figure}

The idea of the proof is the following. Given a low-distortion embedding of $P_{k,n}$ into $\ell_1^d$, we naturally obtain
a mapping that maps each edge of the graph $G_{k,n}$ to a $d$-dimensional vector (namely, the difference between the two embedded endpoints)
whose $\ell_1$ norm is close to $1$. Assume for simplicity that this norm is exactly $1$;
assume moreover that the vector has non-negative coordinates. (In the proof we will show
how to reduce the general case to this case.) So we can equivalently view this mapping as an encoding from $[2k]^n$ to probability
distributions over $[d]$. Using the second property mentioned above,
one can obtain the following crucial property of the encoding: For any $\ell \in [n]$
and any $x_1,\ldots,x_{\ell-1} \in [2k]$, if we are given $x_1,\ldots,x_{\ell-1}$
together with the encoding of $(x_1,\ldots,x_{n}) \in [2k]^n$, where $x_\ell,\ldots,x_n$ are chosen uniformly,
then we have a good probability to guess $x_\ell~{\mod}~k$ (perfect probability in case of no distortion). A basic information theoretic
argument now provides a lower bound on $d$ of any such encoding.
For instance, in the case there is no distortion, the encoding allows us to predict
$x_\ell~{\mod}~k$ as above with certainty, and the information theoretic argument gives
the tight bound $d \ge k^n$.
%, which matches the natural (deterministic) encoding that stores $(x_1,\ldots,x_n)~{\mod}~k$.
We note that this simple yet powerful information theoretic argument appears in various different contexts,
such as that of quantum random access codes~\cite{nayak:qfa}.

\section{Preliminaries}

All logarithms are base $2$. We use $[k]$ to denote the set $\{1,\ldots,k\}$.
We now list a few basic definitions and facts from information theory.
Although not really needed for our proof, the interested reader can find
an introduction to the area in~\cite{CoverT06}.
We let $H(\delta):=-\delta \log \delta - (1-\delta) \log(1-\delta)$ denote the \emph{binary entropy function}.
For a random variables $X$ on a domain $[d]$ obtaining each value $i \in [d]$ with probability $p_i$,
the \emph{entropy} of $X$ is given by $H(X):=-\sum_i p_i \log p_i$, and is always at most $\log d$.
For two random variables $X,Y$, the \emph{conditional entropy} $H(X~|~Y)$ is the expectation
of $H(X~|~Y=y)$ over $y$ chosen according to $Y$; this can be seen to equal $H(XY) - H(Y)$.
Finally, the \emph{mutual information} $I(X:Y)$ is defined as $H(X)+H(Y)-H(XY)=H(X)-H(X|Y)$, and the
\emph{conditional mutual information} $I(X:Y~|~Z)$ is the expectation of $I(X:Y~|~Z=z)$ over
$z$ chosen according to $Z$, or equivalently, $H(X~|~Z)+H(Y~|~Z)-H(XY~|~Z)$. The \emph{data processing inequality} says that applying
a function cannot increase mutual information, $I(f(X):Y) \le I(X:Y)$.

The following claim (which is essentially what is known as Fano's inequality) shows that if one random variable can be used to predict
another random variable, then their mutual information cannot be too small.
\begin{claim}\label{clm:fano}
Assume $X$ is a random variable uniformly distributed over $[k]$. Let $Y$ be another
random variable, and assume that there exists some function $f$ with range $[k]$ such that $f(Y)=X$
with probability at least $p \ge 1/2$. Then $I(X:Y) \ge \log k - (1-p) \log(k-1) - H(p)$.
\end{claim}
\begin{proof}
By the data processing inequality,
$$ I(X:Y) \ge I(X:f(Y)) = H(X) - H(X~|~f(Y)) = \log k - H(X~|~f(Y)),$$
so it suffices to bound $H(X~|~f(Y))$ from above.
Since conditioning cannot increase entropy,
\begin{align*}
H(X~|~f(Y)) &= H(1_{X=f(Y)},X~|~f(Y)) \\
&= H(1_{X=f(Y)}~|~f(Y)) + H(X ~|~ 1_{X=f(Y)},f(Y)) \\
&\le H(1_{X=f(Y)}) + H(X ~|~ 1_{X=f(Y)},f(Y)) \\
&\le H(p) + (1-p) \log(k-1).\qedhere
\end{align*}
\end{proof}

\section{Proof}

Our main technical theorem is the following.

\begin{theorem}\label{thm:main}
For any $k \ge 2$, $n \ge 1$ the following holds. Assume $f:[2k]^n \to \R^d$ satisfies
that for all $x_1,\ldots,x_n \in [2k]$, $\|f(x_1,\ldots,x_n)\|_1 \le 1$ and, moreover, that
for some $\eps < 1/(k-1)$, and for all $\ell \in [n]$, $x_1,\ldots,x_{\ell-1} \in [2k]$, and $r \in [k-1]$,
\begin{align}
  &\frac{1}{2k} \left\| \sum_{b=1}^r (f(x_1,\ldots,x_{\ell-1},b) + f(x_1,\ldots,x_{\ell-1},b+k)) - \right. \nonumber \\
    &\left. \qquad \qquad \sum_{b=r+1}^k (f(x_1,\ldots,x_{\ell-1},b) + f(x_1,\ldots,x_{\ell-1},b+k)) \right\|_1 \ge 1-\eps\label{eq:constr}
\end{align}
where $f(x_1,\ldots,x_\ell)$ denotes the average of $f(x_1,\ldots,x_n)$ over $x_{\ell+1},\ldots,x_n$ chosen
uniformly in $[2k]$. Then
\begin{align}\label{eq:lowerbound}
d \ge 2^{(\log k - \delta \log (k-1) - H(\delta))n - 1} - \frac{1}{2},
\end{align}
where $\delta := (k-1)\eps/2 < 1/2$.
\end{theorem}

Before proving the theorem, let us explain how it implies Theorem~\ref{thm:mainthm}.
Consider any embedding $F$ of $P_{k,n}$ into $\ell_1^d$ with distortion at most $1/(1-\eps)$ for some $\eps < 1/(k-1)$.
By scaling $F$, we can assume that it is $1$-Lipschitz (i.e., it does not expand any distance) and that distances
are not contracted by more than $1-\eps$. Let $f$ be the function that maps each
$x \in [2k]^n$ to $F(u)-F(v)$, where $u$ is the label of the right endpoint of the edge labeled by $x$
and $v$ is the label of its left endpoint. Since $F$ is $1$-Lipschitz, $\|f(x)\|_1 \le 1$ for all
$x \in [2k]^n$. Moreover, it is not difficult to see that $f$ satisfies Eq.~\eqref{eq:constr} (see Figure~\ref{fig:graph2}).
Hence, Theorem~\ref{thm:main} implies that the bound in Eq.~\eqref{eq:lowerbound} holds.

\begin{figure}
\begin{centering}
\ifpdf
\includegraphics[width=0.3\textwidth]{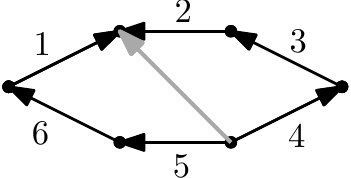}
\fi
\caption{The condition in Eq.~\eqref{eq:constr} for $r=1$, $k=3$.}
\label{fig:graph2}
\end{centering}
\end{figure}

For the first part of Theorem~\ref{thm:mainthm} we fix $k=2$. We obtain that for any $D \ge 1$,
any distortion-$D$ embedding of $G_{2,n}$ (so $\eps=1-1/D$ and $\delta=1/2-1/(2D)$) must have dimension at least
$$2^{(1 - H(1/2-1/2D))n-1}-\frac{1}{2} = 2^{\Omega(n/D^2)} = N_{2,n}^{\Omega(1/D^2)}.$$
For the second part of Theorem~\ref{thm:mainthm}, choosing $k\approx 1/(\eps \log(1/\eps))$ and noting that $\delta \log k = O(1)$, we obtain that the dimension must be at least
$$(2k)^n 2^{(-\delta\log k - 2)n - 1} - \frac{1}{2} =  N_{k,n}^{1-O(1/\log(1/\eps))}.$$

\begin{proof}[Proof of Theorem~\ref{thm:main}]
We start by considering the case that
for all $x_1,\ldots,x_n \in [2k]$, $f(x_1,\ldots,x_n)$ has non-negative coordinates and $\ell_1$-norm $1$.
We will later see how this implies the general case. Making this assumption allows us to think
of $f(x_1,\ldots,x_n)$ as a probability distribution over $[d]$.
Let $X=(X_1,\ldots,X_n)$ and $M$ be two random variables where $X$ is uniformly distributed over $[2k]^n$
and $M$ is distributed over $[d]$ according to $f(X)$. Using the chain rule for mutual information we obtain
$$ \log d \ge H(M) \ge I(X:M) = I(X_1:M) + I(X_2:M ~|~ X_1) + \cdots + I(X_n:M ~|~ X_1,\ldots,X_{n-1}).$$
The following lemma implies that for any $\ell \in [n]$,
$$ I(X_\ell:M ~|~ X_1,\ldots,X_{\ell-1}) \ge \log k - \delta \log (k-1) - H(\delta) $$
(this is true even conditioned on any fixed value of $X_1,\ldots,X_{\ell-1}$, and not just on average)
and therefore
$$ d \ge 2^{(\log k - \delta \log (k-1) - H(\delta))n }.$$

\begin{lemma}
Let $A$ and $B$ be two random variables such that $A$ is uniformly distributed over $[2k]$
and for any $a \in [2k]$, conditioned on $A=a$, $B$ is distributed according to some probability distribution $P_a$
on $[d]$. Assume that for all $r \in [k-1]$,
\begin{align*}
  \frac{1}{2k} \left\| \sum_{a=1}^r (P_a + P_{a+k}) - \sum_{a=r+1}^k (P_a + P_{a+k}) \right\|_1 \ge 1-\eps.
\end{align*}
Then $I(A:B) \ge \log k - \delta \log (k-1) - H(\delta)$.
\end{lemma}
\begin{proof}
Let $A' = ((A-1)~{\mod}~k)+1$, and notice that $A'$ is uniformly distributed on $[k]$. By the data processing inequality, $I(A:B) \ge I(A':B)$.
For any $a \in [k]$, let $Q_a := (P_a + P_{a+k})/2$ be the distribution of $B$ conditioned on $A'=a$.
Our assumption says that for all $r \in [k-1]$,
$$
  \frac{1}{k} \left\| \sum_{a=1}^r Q_a - \sum_{a=r+1}^k Q_a \right\|_1 \ge 1-\eps.
$$
We need the following easy claim.
\begin{claim}
For any $p_1,\ldots,p_k \ge 0$,
$$\left(\sum_{i=1}^k p_i \right) - \max\{p_1,\ldots,p_k\} \le \frac{1}{2} \sum_{r=1}^{k-1} \left(\left(\sum_{i=1}^k p_i\right) - \left| \sum_{i=1}^r p_i - \sum_{i=r+1}^k p_i\right|\right).$$
\end{claim}
\begin{proof}
Let $r^* \in \{0,\ldots,k-1\}$ be the largest such that the expression inside the absolute value is negative.
Then the sum of the absolute values at $r=r^*$ and $r=r^*+1$ is exactly $2 p_{r^*+1}$. The claim follows.
\end{proof}
By applying the inequality to each of the $d$ coordinates of the probability distributions $Q_a$, and summing
the results, we obtain
$$
1 - \frac{1}{k} \|\max\{Q_1,\ldots,Q_k\}\|_1
   \le \frac{1}{2} \sum_{r=1}^{k-1} \left(1 - \frac{1}{k} \left\| \sum_{a=1}^r Q_a - \sum_{a=r+1}^k Q_a \right\|_1\right)
$$
and hence
$$ \frac{1}{k} \|\max\{Q_1,\ldots,Q_k\}\|_1 \ge 1 - (k-1)\eps/2 = 1-\delta. $$
Consider the function that maps each $j \in [d]$ to the $a \in [k]$ that maximizes $\Pr[Q_a = j]$.
This function correctly predicts $A'$ from $B$ with probability
$\frac{1}{k} \|\max\{Q_1,\ldots,Q_k\}\|_1$. The lemma now follows from Claim~\ref{clm:fano}.
\end{proof}

We now show how to derive a similar bound for any $f$ as in the statement of the
theorem. Let $f:[2k]^n \to \R^{d}$ be such that for all $x \in [2k]^n$,
$f(x)$ has $\ell_1$ norm at most 1. Define $g:[2k]^n \to \R^{2d+1}$
by the concatenation
$$g(x) := \max\{f(x),0\} ~.~ \max\{-f(x),0\} ~.~ 1 - \|f(x)\|_1.$$
Obviously, for all $x$, $g(x)$ is non-negative and has $\ell_1$ norm $1$.
Moreover, the linear operator that maps any $y \in \R^{2d+1}$ to the vector $(y_j-y_{j+d})_{j=1}^d \in \R^d$
cannot increase the $\ell_1$ norm and maps $g(x)$ to $f(x)$ for all $x$.
Therefore Eq.~\eqref{eq:constr} holds for $g$, and the theorem follows.
\end{proof}

\subsection*{Acknowledgments}

I thank the organizers of the workshop ``Metric embeddings, algorithms and hardness of approximation"
in the Institut Henri Poincar{\'e}, where this work started. I also thank Moses Charikar for the
inspiring talk he gave there, and Assaf Naor and Ofer Neiman for useful discussions.

\end{document}